\documentclass[10pt]{article}
\usepackage{amssymb,amsfonts}
\usepackage{pb-diagram}
\usepackage{amsmath,amsthm,amsfonts,amssymb}
\usepackage[usenames]{color}
\usepackage{hyperref}
\usepackage{soul}
\usepackage{graphicx}
\newtheorem{theorem}{Theorem}
\newtheorem{definition}{Definition}

\newtheorem{lemma}{Lemma}

\newtheorem{prop}{Proposition}

\newcommand{\be}{\begin{enumerate}}
\newcommand{\ee}{\end{enumerate}}
\newcommand{\beq}{\begin{equation}}
\newcommand{\eeq}{\end{equation}}
\usepackage{amssymb}
\usepackage{amsmath}
\usepackage{amssymb}
\usepackage{imakeidx}

\newcommand{\comment}[1]{} 

\begin{document}

\title{First-order sentences in random  groups III }
\author{Olga Kharlampovich, Alexei Miasnikov, Rizos Sklinos}

\maketitle
\begin{abstract} We show that a first-order sentence is almost surely true in a random group of density $d<1/2$ if and only if it is true in a non-abelian free  group.
\end{abstract}
\section{Introduction}
It was proved in \cite{Sela5.2},\cite{KMel} that every formula in the theory of a free group $F$ is equivalent to a boolean combination of $\exists\forall$-formulas.  In \cite{KMel} it is also proved that the elementary theory of a free group is decidable (there is an algorithm given a sentence to decide whether this sentence belongs to $Th(F)$ ).
If the language of the free group $F$ contains its generators as constants, then it was shown in \cite{KMdef} that
 every definable subset of $F$ is defined by some boolean combination of
formulas 
 \begin{equation}\label{AE113}
\exists x\forall y (\Sigma (a,p,x)=1\wedge \Psi (a,p,x,y)\neq 1),
\end{equation}where $a$ are constants, $x,y,p$ are tuples of variables. 

Let $F_n$ be a  free  group with generators $e_1,\ldots ,e_n$, and $\Gamma$ be a random group in Gromov's density model for $d<1/2$
(for each $\ell$, $\Gamma _{\ell}$ is a quotient of $F_n$ over a random set of relations of density $d$). We proved in \cite{KhS} that a universal sentence is true in $\Gamma$ (for density $d<1/16$) if and only if it is true in a non-abelian free group. In \cite{KhS1} we extended this result to boolean combinations of
$\forall\exists$ sentences. It was noticed in \cite{Mass1}  that using the method of \cite{KhS} one can prove the statement about universal sentences for $d<1/2$, see the appendix at the end of this paper. This and \cite{KhS1} imply that a boolean combinations of $\forall\exists$ sentences holds in a random group $\Gamma$ at  $d<1/2$ if and only if it holds in a non-abelian free group.  In \cite{Mass2} it was proved that a first-order sentence of minimal rank is true in $\Gamma$ if and only if it is true in a non-abelian free group.

Here we will present the  proof of the following results.
\begin{theorem} \label{main}  
A first-order formula  is equivalent to  a boolean combination  of $\exists\forall$-formulas in $F_n$ if and only if it is almost surely equivalent to this boolean combination in $\Gamma$. 
\end{theorem} The idea of the proof is to follow the quantifier elimination procedure used in the proof of the Tarski conjecture - that the theories of non-abelian free groups of different ranks coincide - and to observe that at each step, the result concerning universal sentences guarantees that the procedure proceeds identically for $F_n$ and $\Gamma$.
This theorem and \cite{KhS1} imply
\begin{theorem}
 A first-order sentence is almost surely true in a random group of density $d<1/2$ if and only if it is true in a non-abelian free group.   
\end{theorem}

We will use the machinery developed for answering Tarski's question and in particular {\em formal solutions, towers, closures of towers} (see \cite{Imp}, \cite{Sela2}) and the {\em process of validating a $\forall\exists$-sentence} (see \cite{Sela4}, \cite{KMel}). 
The reader should be familiar with \cite[Section 2]{KhS1} where necessary definitions are given.

\section{Preliminary results} 
We say that a property $P$ holds in a random group $\Gamma$ if the probability that $\Gamma _{\ell}$ has $P$ approaches 1, as $\ell$ approaches infinity. 

An ascending sequence of groups in density $d$ is a sequence of groups $\{\Gamma _{\ell _n}\}$ such that $\ell _n\rightarrow\infty$ and $\Gamma _{\ell _n}$ is a random group with relations of length $\ell _n$ and density $d$.

Let \( \{h_{\ell_n} : F_m \to \Gamma _ {\ell_n}\} \) be a sequence of homomorphisms over an ascending sequence of groups at $d<1/2$. Assume that the stretching factors 
\[
\mu_n = \min_{f_n \in \Gamma_{\ell _n}} \max_{1 \leq u \leq m} d_{\Gamma _{\ell _n}} (1, \tau_{f_n}\circ h_{ln}(x_u)),
\]
satisfy that \( \mu_n \geq \ell_n^2 \) for all \( n \), where \( X_n \) is the Cayley graph of \( \Gamma _{\ell_n} \) with respect to the generating set \( a \). Then, the sequence 
\[ 
\tau_{f_n} \circ h_{\ell_n},
\]
where \( \tau_{f_n} \) is the corresponding inner automorphism of \( \Gamma_n \) (see \cite[Section 9]{Mass1}), subconverges to an isometric non-trivial action of \( F_m \) on a pointed real tree \( (Y, y_0) \).

\begin{theorem} \label{prelim} 1. Solution set of a system of equations over $F_n$ and $\Gamma$ is encoded in the same M-R diagram.

2.  Every system of equations $U (x)=1$ over $\Gamma$ is equivalent to the same  finite subsystem $V(x)=1$ as for $F_n$. Namely, every solution of $ V(x)=1$ in $\Gamma$ is a solution of  $U (x)=1$ with overwelming probability (w.o.p).

3.  For every convergent sequence of homomorphisms $h_{\ell}: G=\langle x | V(x)\rangle \rightarrow  \Gamma _{\ell}$ the limit of this sequence is a limit group (over $F_n$).

4. Let $K$ be a solid limit group over $H\leq K$. Let $N$ be the bound on the number of strictly solid homomorphisms $K\rightarrow F_n$.  Then $N$ is the bound on the number of strictly solid homomorphisms $K\rightarrow \Gamma _{\ell}$ w.o.p.  Moreover,  for large $\ell$, for any  fixed homomorphism $H\rightarrow \Gamma _{\ell}$ the pre-images of representatives of distinct strictly solid families  $K\rightarrow \Gamma _{\ell}$ w.o.p.are strictly solid homomorphisms from distinct families.

5. A complete system of reducing (=flexible) quotients of a  solid limit group modulo a subgroup is the same for $F_n$ and $\Gamma$.
\end{theorem}

\begin{proof} 1. The solution set of a system of equations over a free group $F_n$ is encoded by a M-R diagram. Since for a given system of equations the solution over $\Gamma$ is almost surely the same as over $F_n$ (we say that it can be lifted into $F_n$), it is encoded by the same M-R diagram. 

2. Let $U(x)=1$ be an infinite system equations and $V(x)$ a finite subsystem equivalent to it in $F_n$. Every solution $\bar b$ of $V(x)$ in $\Gamma$ almost surely has a pre-image in $F_n$ that is a solution of $V(x)$ and, therefore, a solution of $U(x)$. Therefore $\bar b$ is a solution of $U(x)$ in $\Gamma$.

3. We can order the sets of all finite systems of equations $W_1\subset W_2\subset\ldots $. Let $p_{\ell}$ be the probability that $\Gamma _{\ell}$  every solution of every system in $W_{\ell}$ is a free solution (can be lifted to $F_n$). Then $p_{\ell}\rightarrow 1$  as $\ell\rightarrow\infty$. A consequence of this is our statement.

4. The property to have at most $N$ strictly solid families for a solid limit group  can be described by a boolean combination of $\forall\exists$-sentences. Now the statement follows from the main result of \cite{KhS1}. 

One can also give a direct proof. Assume, by contradiction, that there exist more than \( N \) pairwise distinct 
\(\Gamma\)-strictly solid families. Let \( h_j \) be strictly solid representatives for these families. 
For each \( j \), we lift the homomorphism \( h_j \) into a homomorphism 
\[
\tilde{h}_j : K \rightarrow F_k.
\]
Then, since the homomorphisms \( h_j \) represent pairwise distinct 
\(\Gamma\)-strictly solid families, the homomorphisms \( \tilde{h}_j \) must represent 
pairwise distinct strictly solid families. This gives a contradiction.

5. This follows from the fact that homomorphisms factoring through the reducing quotients can be described  by a disjunction of systems of equations.    
\end{proof}
We will call strictly solid (and rigid) solutions {\em algebraic solutions} (similarly to the fields terminology).

 If the obtained limit of  \( \{h_{\ell_n} : F_m \to \Gamma _ {\ell_n}\} \) (that is, by the theorem above, is a limit group), is freely indecomposable, then the homomorphisms $h_{\ell _n}$ can be eventually shortened by precomposing with an automorphism of the limit group. If it is freely decomposable, we can apply automorphisms to each free factor.

Recall that we say that a resolution is graded relative (or modulo) some subgroups of the coordinate group  if these subgroups are elliptic in the JSJ decompositions on all levels in the construction of this resolution. A graded test sequence (=generic family), with respect to an algebraic  family of solutions of the terminal group of completion of the resolution, is a test sequence of the ungraded tower obtained by resolving the graded tower using that family.  We will use a definition of the {\em graded test sequence over random groups} \cite[Def. 3.5]{Mass2}.  The  difference from a graded test sequence over a  free group is that instead of  an algebraic family of solutions in $F_n$ in the bottom of the test sequence over a  free group, we use a family of algebraic solutions of the terminal solid group of the completion into a sequence of groups \(\{\Gamma _{\ell _n}\}\).

We refer to \cite[Section 5]{Mass2} For the construction of the {\em Graded formal M-R diagram} for $F_n$ and for $\Gamma$. The main result is that this diagram of the system $\Sigma(a,p,x,y) = 1$ over a graded resolution $\text{Res}(a,p,x)$ over a free group encodes w.o.p. all the formal solutions $y=f(a,x)$ over $\Gamma$  for any value $p_0\in \Gamma  $ for which 
$$\forall x\in \text{Res}(a,p_0,x) \Sigma(a,p_0,x,y) = 1  $$
\begin{theorem}\label{5.10}\cite[Theorem 5.10]{Mass2} Let
\[
\forall x \exists y \Sigma(a,p,x,y) = 1 \wedge \Psi(a,p,x,y) \neq 1,
\]
be an $\forall\exists$-formula. Let \( \text{Res}(a,p,x) \) be a graded resolution, and denote its completion by \( \text{Comp}(a,p,x,x^{(1)}) \).  
Denote by \( B \) the p-solid base group of \( \text{Comp}(a,p,x,x^{(1)}) \). Let \( \text{GFMRD} \) be the graded formal MR diagram of the system \( \Sigma(a,p,x,y) \) with respect to the resolution \( \text{Res}(a,p,x) \), that was constructed over the free group \( F_n \).

Let \( \Gamma_l \) be a group of level \( l \) in the model. If \( l \) is large enough, then the following property is satisfied in \( \Gamma_l \):

Every sequence \( \{(a, p_n, x_n, {x^{(1)}}_n, y_n)\} \) of \( \Gamma_l \)-specializations for which the restricted sequence \( \{(a, p_n, x_n, {x^{(1)}}_n)\} \) is a \( \Gamma_l \)-test sequence of \( \text{Comp}(a,p,x,x^{(1)}) \), and \( \Sigma(a, p_n, x_n, y_n) = 1 \) in \( \Gamma_l \) for all \( n \), factors through \( \text{GFMRD} \).

\vspace{0.5cm}

Moreover, let \( p_0 \in \Gamma_l \), and let \( h_{p_0} \) be a \( \Gamma_l \)-algebraic solution of \( B \) that maps \( p \) to \( p_0 \). Denote by \( \text{Res}(a, p_0, x) \) the ungraded resolution over \( \Gamma_l \) obtained from \( \text{Res}(a,p,x) \) by specializing its terminal group according to \( h_{p_0} \), and the corresponding \( \Gamma_l \)-ungraded completion by \( \text{Comp}( a,p_0,x,x^{(1)}) \).

Assume that the sentence:
\[
\forall x \in \text{Res}(a, p_0, x) \exists  y\Sigma(a, p_0, x, y) = 1 \wedge \Psi(a,p_0, x,y) \neq 1,
\]
is a truth sentence over \( \Gamma_l \).

Then, there exist some resolutions \( \text{GFRes}_1(a, p_n, x_n, {x_n}^{(1)}, y_n), \dots, \text{GFRes}_m(a, p_n, x_n, {x_n}^{(1)}, y_n) \) in \( \text{GFMRD} \), with corresponding terminal graded formal closures $$ \text{GFCl}_1(a,p,x,x^{(1)},y), \dots, \text{GFCl}_m(a,p,x,x^{(1)},y), $$ together with some \( \Gamma_l \)-algebraic solutions \( g_{p_0}^1, \dots, g_{p_0}^m \) for the p-rigid or p-solid base groups \( \text{Term}(\text{GFCl}_1), \dots, \text{Term}(\text{GFCl}_m) \), so that:

\begin{enumerate}
    \item For all \( i = 1, \dots, m \), the solution \( g_{p_0}^i \) restricts to a \( \Gamma_l \)-algebraic solution of \( B \) that belongs to the same \( \Gamma_l \)-strictly solid family of \( h_{p_0} \).
    \item Every sequence of \( \Gamma_l \)-specializations \( \{(a, p_n, x_n, {x_n}^{(1)}, y_n)\} \), for which \( \{(a, p_n, x_n, {x_n}^{(1)})\} \) is a \( \Gamma_l \)-test sequence through \( \text{Comp}(a,p_0, x, x^{(1)}) \), and \( \Sigma(a, p_n, x_n,  y_n) = 1 \) and \( \Psi(a, p_n, x_n, y_n) \neq 1 \) in \( \Gamma_l \) for all \( n \), factors through one of the resolutions \( \text{GFRes}_1(a,p_0, x, x^{(1)},y), \dots, \text{GFRes}_m(a,p_0, x, x^{(1)},y) \).
    \item The induced \( \Gamma_l \)-ungraded formal closures \( \text{FCl}_1(a,p_0, x, x^{(1)},y), \dots, \text{FCl}_m(a,p_0, x, x^{(1)},y) \) form a \( \Gamma_l \)-covering closure for \( \text{Comp}(a,p_0, x, x^{(1)}) \).
    \item For all \( i = 1, \dots, m \), there exists a formal solution \( y_i = y_i(a,p,x, x^{(1)}, s) \in \text{GFCl}_i \) that factors through the resolution \( \text{GFRes}_i(a,p, x, x^{(1)},y) \), so that the words \( \Sigma(a,p, x,y_i) \) represent the trivial element in \( \text{GFCl}_i \), and each of the words \( \Psi(a,p_0, x,y_i) \) is non-trivial in \( \text{FCl}_i(a,p_0, x, x^{(1)},y) \).
\end{enumerate}
\end{theorem}

\section{Validation of $\forall\exists$-sentence in a random group} 
We proved in \cite{KhS1} that $\forall\exists$-sentence in true in a random group if and only if it is in the theory of the free group.  The process of validation of $\forall\exists$-sentence in a random group is the same as in the free group.

\section{Quantifier elimination}\label{sec:6}
In this section we will prove Theorem \ref{main}. Consider the following formula

\begin{equation}\label{38neg} \Theta(p)=\exists z \forall x\exists y (\Sigma(a,p,z,x,y)=1\wedge \Psi (a,p,z,x,y)\not =1),\end{equation}
where $a$ is a basis of $F_n=F(a)$.

To obtain quantifier elimination to boolean combinations of $\exists\forall$-formulas it is enough to find such a boolean combination that defines the set defined by  $\Theta(p).$
\subsection{Parametric $\forall\exists$-tree over a free group $F$ and a random group $\Gamma$}
We follow \cite[Section 3.5]{KMbook}. The parametric $\forall\exists$-tree over $F_n$ encodes all finitely many {\em proof systems}.  We will describe the construction and show that the tree is the same for $F_n$ and $\Gamma$.  The branches of the tree consist of alternating graded MR diagrams and graded formal MR diagrams (describing formal solutions obtained  using the Implicit function theorem). We apply formal solutions to our fixed system $\Sigma (a,p,z,x,y)=1$ on the  specializations that factor through a graded resolution in the graded MR diagram that was placed in the previous step in the proof system. Those specializations (in $\Gamma$ and in $F_n$) for which all of the formal solutions applied to them do not satisfy our fixed inequation $\Psi (a,p,z,x,y)\neq 1$ (the remaining y's are passed to the next graded MR diagram in the proof system (for such specializations in $\Gamma$, there is al least on pre-image in $F$ that is passed to the next level).

For every tuple of elements $p_0$ for which $\Theta(p_0)$ is true, there exists some $z$ and (by the Merzljakov theorem  \cite{KMel}, Theorem 4, a solution $Y=f(a,p_0,z_0,x)$ of $\Sigma=1\wedge \Psi\neq 1$ in $F(x)\ast F.$ All formula solutions of $\Sigma=1$ for all possible values of $p$ belong to a finite number of resolutions with terminal groups $F _{R(U_{1,i})}\ast F(x),$ where $U_{1,i}=U_{1,i}(a,p,z,z^{(1)})$ and $F _{R(U_{1,i})}$ is a solid group modulo $\langle a,p,z\rangle$ (see Section 12.2, \cite{KMel}). The groups $F _{R(U_{1,i})}\ast F(x)$ are terminal groups of the MR-diagram for $\Sigma(a,p,z,x,y)=1$ modulo $\langle a,p,z\rangle $.

We now consider each of these resolutions separately. Those values $p,z$ for which there exist a value of $x$ such that the equation 
$$\Psi(p,z,x, f(z, z^{(1)},p, x))=1$$ is satisfied for any function $f$ give a system of equations on $F _{R(U_{1,i})}\ast F(x).$ This system is equivalent over $F_n$ and over a random group  $\Gamma$ to the same finite subsystem by Theorem \ref{prelim} (to one equation in the case when we consider formulas with constants).  Let $G$ be the coordinate group of this system and  $G_i, i\in J$ be the corresponding limit groups.

 The parametric  $\forall\exists$-tree,   $T_{AE}( \Theta ; p,z)$ is constructed  the same way as the $\forall\exists$-tree with $x, y$ considered as variables and $p,z$ as parameters.  To each group $G_i$  we assign resolutions modulo $\langle p,z\rangle$. Their terminal groups are groups  $F _{R(V_{2,i})}$, where
  $$V_{2,i}=V_{2,i}(p,z,z^{(1)}, z_1^{(2)})$$
   are solid limit groups modulo  $\langle a,p,z\rangle$. Then we find all formula solutions $y$ of the equation $$\Sigma (a,p,z,x,y)=1$$ in the closures of the towers corresponding to these resolutions for $x$ (see  \cite[Theorem 12]{Imp}). By this or by Theorem \ref{5.10}, these formula solutions $y$ are described by a finite number of resolutions with terminal groups $F_{R(U_{2,i})},$ where $U_{2,i}=U_{2,i}(p,z,z^{(1)},z_{1}^{(2)}, z^{(2)}).$ Here $F_{R(U_{2,i})},$ are solid with respect to $\langle p,z\rangle$ (see the proof of \cite[Lemma 5.2]{Mass2} with $F _{R(V_{2,i})}$ in place of $B$ and  $F_{R(U_{2,i})},$ in place of $Term (M)$). 
   
   Then again we investigate the values of $x$ that make the word $\Psi (a,p,z,x,y)$ equal to the identity for all these formula solutions $y$. And we continue the construction of $T_{AE}( \Theta ; p,z)$.
   
   There is the following difference in the construction of this parametric $\forall\exists$-tree and the construction of the $\forall\exists$-tree in \cite{KhS1}. Instead of using  shortest form morphisms,  as in the definition of shortest form morphisms in \cite[Definition 2.28]{KhS1}  in corresponding places we use minimal strictly solid solutions.  We can prove that this tree is finite exactly the same way  the finiteness of the $\forall\exists$-tree is proved.   For each branch of the  tree $T_{AE}( \Theta ; p,z)$  we assign a
 sequence of limit groups $$F _{R(U_{1,i})},
F _{R(V_{2,i})}\ldots ,F _{R(V_{r,i})}, F _{R(U_{r,i})}$$
that we have just constructed,   as in \cite[Section 12.2]{KMel}. 
\begin{lemma}
The construction of the  tree $T_{AE}( \Theta ; p,z)$  for $F_n$ and a random group $\Gamma$ is the same w.o.p.
\end{lemma}
\begin{proof} This follows from the following two facts. 
 Non-taut solutions in $\Gamma _{\ell}$ almost surely have some non-taut pre-images (See definition of a taut morphism \cite[Definition 2.27]{KhS1}. If $K$ is a solid limit group relative to $H$, then flexible homomorphisms $K\rightarrow \Gamma _{\ell}$ almost surely have some flexible pre-images.
    
\end{proof}
 Corresponding irreducible systems of equations are: $$U_{1,i}=U_{1,i}(a,p,z,z^{(1)})=1,$$ which correspond to the terminal
groups of resolutions  describing $y$  of level
$1$,
 $$V_{m,i}=V_{m,i}(a,p,z,z^{(1)}, z_1^{(m)})=1,\ m=2,\ldots ,r$$
which correspond to the terminal groups of resolutions describing remaining $X$
 of level $m-1$ and 
$$U_{m,i}=U_{m,i}(a,p,z,z^{(1)},
z_1^{(m)}, z^{(m)})=1,\ m=2,\ldots ,r,$$ which correspond to the terminal
groups of resolutions describing $y$  of level
$m, m>1$, 
Systems $V_{m,i}=1$ correspond to vertices of  $T_{AE}( \Theta ; p,z)$ 
that have distance $m$ to the root.

For each $m$ the groups $F _{R(U_{m,i})}$ and $F _{R(V_{m,i})}$ are solid modulo the subgroup $\langle p,z\rangle .$ $F _{R(U_{m-1,i})}$ is mapped into $F _{R(V_{m,i})}$ compatibly. That is, the minimal subgraph of groups of the $p,z$-JSJ of $F _{R(V_{m,i})}$
containing the canonical image of $F _{R(U_{m-1,i})}$, consists only from parts that are taken from the $p,z$-JSJ of $F _{R(U_{m-1,i})}$.

 Below we will sometimes skip index $i$ and write $U_m=1,\
V_m=1$ instead of $U_{m,i}=1,\ V_{m,i}=1.$

\subsection{Configuration groups}

In order to show that a specialization $p_0$ of the parameters $p$ is in the set making the formula true, $True(\Theta)$, one needs to find a specialization $ z_0$ of variables $z$ such that the corresponding $\forall\exists$-sentence
$$ \forall X\exists y (\Sigma (p_0, z_0,x,y)=1\wedge \Psi(p_0,z_0,x,y)\not =1)$$
 is true.  The proof that this sentence is true corresponds to a subtree of $T_{AE}( \Theta ; p,z)$ (It is called in \cite{KMel} a true-subtree for $p_0,z_0$, Sela calls
each possibility for the structure of a proof encoded in the parametric $\forall\exists$-tree
{\em a proof system}.
). That is the proof consists of a finite sequence of formal solutions $y$ in closures of corresponding towers in the vertices of this subtree.  Since $T_{AE}( \Theta ; p,z)$ is a finite tree, there is only a finite number of possibilities for  such true-subtrees of $T_{AE}( \Theta ; p,z)$.  Note that given $ p_0, z_0$ there may be several 
 true-subtrees for this tuple, but the number of possible true-subtrees is bounded. We will say that the sentence  associated with the tuple 
 $p_0, z_0$  can be proved at level $m$ if the maximal depth of the true-subtree is $m$.
 
\begin{definition}
 Denote by $True(\Theta)_i$  ($True(\Theta )_{i \Gamma}$) the set of specializations $ p_0$ of $p$  for which there exists $z_0$ such that the corresponding $\forall\exists$-sentence can be proved on level $i$ (there are solutions to some $U_{i,j}=1$ but there are no remaining $x$ on level $i$ that is equivalent to the absence of solutions  to $V_{i+1,j}=1$).
\end{definition}

\begin{lemma} The set $True(\Theta)_1$ ($True(\Theta )_{1 \Gamma}$) is defined by an  $\exists\forall$-sentence,  the same for $F_n$ and $\Gamma$.
\end{lemma}
\begin{proof}  We can write an $\exists\forall$- formula that says that there exists an index $i$, $z, z^{(1)}$ such that   $U_{1,i}(a,p,z,z^{(1)})=1$,
$z^{(1)}$ is a strictly solid solution, and for any index $j$ there is no specialization $z_1^{(2)}$ such that $V_{2,j}=V_{2,j}(a,p,z,z^{(1)}, z_1^{(2)})=1.$
\end{proof}
\begin{theorem} \label{main11} The set $True(\Theta)_2$ ($True(\Theta )_{2 \Gamma}$) is defined by a formula in the boolean algebra of   $\exists\forall$-sets, that is the same for $F_n$ and $\Gamma$.
\end{theorem}

The rest of the section will be devoted to the proof of this theorem, and at the very end we will discuss the general case of 
$True(\Theta)_m$, where $m$ is bounded by the depth of the tree $T_{AE}( \Theta ; p,z).$
\begin{definition}  \label{d2} Let $p_0\in True(\Theta)_2$. Then there exists a family of  algebraic specializations $(p_0,{z_0}, {z_0}^{(1)})$ for one of the groups $F _{R(U_{1,k})}$.  Let $F _{R(V_{2,1})},\ldots ,F _{R(V_{2,t})}$ be the whole family
of groups on level 1 of the proof constructed for this group $F _{R(U_{1,k})}$ ($k$ is fixed).  To construct the {\em initial
resolutions of level} 2 and {\em width} $i=i_1+\ldots
+i_t$, we consider the resolutions modulo the subgroup
$\langle p\rangle$ for limit groups $H$ discriminated
by $i$ solutions $(p_0,{z_0},{z_0}^{(1)},{(z_0)_1}^{(2,j,s)},{z_0}^{(2,j,s)})$ of the systems
$$U_{2,m_s}(a,p,z,z^{(1)},z_1^{(2,j,s)},z^{(2,j,s)})=1,\ j=1,\ldots ,i_s,\ s=1,\ldots ,t,$$ with the following properties (for some $i_1,\ldots ,i_t$   such  solutions must exist):

(1) $(p_0, {z_0}, {z_0}^{(1)})$ are algebraic (=strictly solid)  solutions of $U_{1,k} (a,p,z,z^{(1)})=1$, 

\noindent
$(p_0, {z_0}, {z_0}^{(1)},(z_0)_1^{(2,j,s)})$ are algebraic solutions of $V_{2,s} (a,p,z,z^{(1)}, z_1^{(2)})=1$,  

\noindent
 $(p_0, {z_0}, {z_0}^{(1)},(z_0)_1^{(2,j,s)},{z_0}^{(2,j,s)})$ are algebraic solutions of $U_{2,m_s} (a,p,z,z^{(1)}, z_1^{(2)}, z^{(2)})=1;$

(2) $ (z_0)_{1}^{(2,j,s)}$ are not equivalent to $ (z_0)_{1}^{(2,p,s)},
p\neq j,\ p,j=1,\ldots, i_s,\ s=1,\ldots ,t$;

(3) for any of the finite number of values of $z_{1}^{(2)}$
the resolutions for $V_{2,s}(a,p
,z,z^{(1)},z_{1}^{(2)})=1$ are contained in the union
of the resolutions for $U_{2, m_s}(a,p
,z,z^{(1)},z_{1}^{(2,j)},z^{(2,j)})=1$ for
different values of $z^{(2,j,s)}$;

(4) there is no non-equivalent $(z_0)_{1}^{(2,i_{s}+1,s)}$, algebraic,
solving $V_{2,s}(p,z,z^{(1)}, z_1^{(2)})=1, s=1,\ldots ,t$.

(5) the solution $(p_0,{z_0},{z_0}^{(1)})$  does not
satisfy a proper equation which for some $i\in J$ implies  $\overline V_i(p_0,{z_0}, x)=1$  for any value of
$x$.

 (6)  for any $s$,  the solution $(p_0
,{z_0},{z_0}^{(1)},(z_0)_{1}^{(2,1,s)},{z_0}^{(2,1,s)})$  can not be
extended to a solution of some
$$V_{3,s}(a,p,z,z^{(1)},z_{1}^{(2,1,s)},z^{(2,1,s)},
z_{1}^{(3,1,s)})=1.$$

We call this group $H$ a {\bf configuration group}\index{Configuration group}. We also call a tuple $$(p_0, {z_0},{z_0}^{(1)},(z_0)_{1}^{(2,j,s)},{z_0}^{(2,j,s)},\ j=1,\ldots i_s,\ s=1,\ldots ,t)$$ satisfying the conditions above {\em a certificate} for $\Theta$ for $p_0$ (of level 2  and width $i$).  A certificate is the same as a valid PS statement \cite[Definition 1.23]{Sela5}. 
 We add to the generators of the configuration group additional variables $q$ for the primitive roots of a fixed set of elements for each certificate (these are primitive roots of the images in $F$ of the edge groups and abelian vertex groups in the relative JSJ decompositions of the groups $F _{R(V_{2,1})},\ldots ,F _{R(V_{2,t})}$). 
 
 Similarly we define a $\Gamma$-certificate for $\Theta$ for $p_0$ over $\Gamma$.\end{definition}
 
  Note that the structure of a certificate depends only on the
proof system, and not on the particular specialization. The (finite) set of configuration groups is the Zariski closure of all the set of certificates. Each configuration group $H$ in this definition is a limit group and a coordinate group of some system of equations, say $$W(p,z,z^{(1)},z_1^{(2,j,s)},z^{(2,j,s)},  q, \ j=1,\ldots i_s,\ s=1,\ldots ,t)=1.$$   We construct for $H$ the canonical taut graded
M-R diagram with respect to the parameter subgroup $P=\langle p\rangle$. The resolutions in this diagram are called PS resolutions. Let $Res(p,z,u)$ be one of resolutions for $H$. Here $u$ includes ${z}^{(1)},(z)_{1}^{(2,j,s)},{z}^{(2,j,s)},\ j=1,\ldots i_s,\ s=1,\ldots ,t$.  When we consider such an initial resolution of level 2 and width $i=i_1+\ldots +i_t$, we have fixed the group
 $F _{R(U_{1,k})}$, and the number $i_s$ of equivalence classes of algebraic solutions for each of the groups $F _{R(V_{2,s})}$  Let $Comp(p,z,u,v)$ be a completion of this resolution.

 The configuration groups are the input for the so called {\em sieve procedure} in \cite{Sela5} (they are the input for the Projective tree in \cite{KMel}).  The goal
of the sieve procedure is to show that the collection of specializations of
the parameters $p$, for which there exists a certificate that
factors through the given configuration group $H$, is in the Boolean algebra of $\forall\exists$
sets.
 
 Not each specialization of $H$
 that has a structure of a certificate (pseudo-certificate)  is, actually, a certificate. We start
the sieve procedure with each of the completions $Comp(p,z,u,v)$  in parallel.
 \begin{lemma} PS resolutions encode all the $\Gamma _{\ell}$ specialisations for which the restriction on  the configuration groups represent $i$ distinct families of algebraic solutions of $V_{2,s} (a,p,z,z^{(1)}, z_1^{(2)})=1$.
 \end{lemma}
 \begin{proof} Any pre-image of such a specialization factors through this M-R diagram.
 \end{proof}
 We can assume that for all $s=1,\ldots ,t$, all the mappings of $F_{V_{2,s}}$ into any level of $Comp(p,z,u,v)$ are compatible with the $pz$-JSJ decomposition of $F_{V_{2,s}}$.

Some generic families of specializations $$(p_0,{z_0},{z_0}^{(1)},(z_0)_1^{(2,j,s)},{z_0}^{(2,j,s)},  q, \ j=1,\ldots i_s,\ s=1,\ldots ,t)$$ in $F_n$ that factor through the fixed  initial completed
resolution $Comp(p,z,u,v)$  (modulo $\langle p\rangle $, with terminal group $Term(p, p^{(1)})$) of level 2 and width $i$, may fail to be certificates for $p_0$ in one of the following ways:

F1.  Specializations in a generic family for  $Comp(p,z,u,v)$ may not satisfy property (1),  instead of being algebraic (=strictly solid) they may become reducing (=flexible), or two not equivalent specializations may become equivalent (failing (2)). These specializations factor through a finite number of  limit groups $Red_1,\ldots ,Red _k$ (extra variables may be added to demonstrate this). By the Implicit Function Theorem the towers for them are closures (=corrective extensions) of the completion  $Comp(p,z,u,v)$. A certificate does not factor through any of 
$Red_1,\ldots ,Red _k$.

F2. Specializations for which property (3) fails  factor through a finite number of limit groups with additional equations involving root variables $Q$, $RQ_1,\ldots ,RQ_m$.

F3.  Specializations may not satisfy property (6), in this case they can be extended to  specializations that  factors through a finite number of limit groups of higher level $F _{R(V_{3,m_s})}$.

Notice that all the specializations that can be extended to specializations that factor through these limit groups $Red_1,\ldots ,Red _k$,  $F _{R(V_{3,m_s})}$, $RQ_1,\ldots ,RQ_m$,  which are closures of the NTQ group for $Comp(p,z,u,v)$, are not certificates.

The system of equations corresponding to $Term(p,p^{(1)})$ has a bounded number of algebraic solutions. If each of these algebraic solutions can be extended to a specialization that factors through the terminal group of one of the corrective extensions above, then no certificate factors through $Comp(p,z,u,v)$. 

In the Cases F1-F3 a pseudo-certificate is not a certificate if it  satisfies an existential positive formula from a given finite family, say $\exists \bar w NC(p,z,u,v, \bar w)=1$. We call specializations  satisfying this formula, collapsed. We have constructed bundles for which, if a generic pseudo-certificate (= generic family of pseudo-certificated) fails to be a certificate, then the whole fiber does not contain certificates. 

F4. In the next bundles it may be that a generic pseudo-certificate is not a certificate, but they still contain a non-generic certificates. Namely, generic family may fail property (4) in the following way. For tuples $p_0,{z_0},{z_0}^{(1)}$ there may be more  non-equivalent algebraic solutions of the equation $V_{2,s}(p,z,z^{(1)}, z_1^{(2)})=1$ than $i_s$. These tuples then factor through one of the groups $H_{surplus}$ discriminated by solutions of $W=1$ together with solutions $z_1^{(2,i_s+1,s)}\rightarrow F$ minimal with respect to fractional Dehn twists. We construct the  well separated graded formal M-R diagram for these generic families \cite[Definitions 2.25-2.27]{KhS1}. This is the {\em Extra PS M-R diagram}. Every resolution of this diagram is  a graded closure of  $Comp(p,z,u,v)$. These graded closures have solid terminal groups  modulo $\langle p\rangle$, that we denote by $Term_1(p,p^{(1)},p^{(2)})$. The group   $Term(p,p^{(1)})$  is mapped compatibly into $Term_1(p,p^{(1)},p^{(2)})$. Namely, the minimal subgraph of groups of the $p$-JSJ decomposition of  $Term_1(p,p^{(1)},p^{(2)})$ consists only from the vertices and edges from the $p$-JSJ of $Term(p,p^{(1)})$.

Now we construct the Extra PS M-R diagram for $\Gamma$. We look at the collection of specializations $\{(p_n,z_n,u_n,v_n, p_n)\}$ over the accending sequence $\{\Gamma _{\ell _n}\}$ for which
the sequence restricts to a generic family of 
$Comp(p,z,u,v)$, specialization  $\{(p_n,z_n,u_n,v_n, p_n)\}$ is not collapsed and there are more non-equivalent algebraic solutions of the equation $V_{2,s}(p,z,z^{(1)}, z_1^{(2)})=1$ than $i_s$.
 We construct the  well separated graded formal MR diagram over $\Gamma$ for this collection of specializations. Terminal groups of the resolutions in this diagram are graded closures of  $Comp(p,z,u,v)$ over $\Gamma$. 
 \begin{theorem} \cite[Theorem 6.8]{Mass2} The same graded formal closures of $Comp(p,z,u,v)$ serve as terminal groups for resolutions in Extra PS MR diagram for $\Gamma$ and for $F_n$.
      \end{theorem}

Some tuples for $Comp(p,z,u,v)$ factoring through $H_{surplus}$ may still be certificates provided they are  going through one of the resolutions for which value of $z_1^{(2,i_s+1,s)}$ is either reducing or equivalent to one of $z_1^{(2,j,s)}, j=1, \ldots ,i_s.$ 
These are {\em Collapsed Extra PS resolutions}.

This analysis shows the following for $F_n$ and $\Gamma$.
\begin{prop} \label{genfam} For each initial completed
resolution of level 2 and width $i$,  for each  $p_0$ factoring through this resolution for which there exists a certificate, there are the following possibilities: 
 \begin{enumerate}\item  there exists 
 a generic family for the resolution for which the restricted specializations are certificates, 
 \item any certificate in this resolution can be extended by $z_1^{(2,i_s+1,s)}$ so that the whole tuple factors through one of the groups $H_{surplus}$, but these variables  $z_1^{(2,i_s+1,s)}$  factor through one of the non-generic resolutions for which $z_1^{(2,i_s+1,s)}$  is either reducing or equivalent to one of $z_1^{(2,j,s)}, j=1, \ldots ,i_s.$\end{enumerate}
In the former case we say that the resolution has depth 1, in the latter case
we will consider resolutions of depth 2  (for level 2 and width $i$).  Similarly there may be constructed resolutions
of larger depth (for level $2$ and width $i$).
\end{prop}

Properties of the Collapse Extra PS Diagram that is the same for $F_n$ and $\Gamma$, are described in \cite[Theorem 7.7]{Mass2}.

For a given value of $p$ the formula $\Theta$
can be proved on level  2 and depth 1
 if and only if  the following
conditions are satisfied.
\begin{enumerate}

\item [(a)] There exist algebraic solutions  for some system of equations corresponding to  the terminal group $Term (p, p^{(1)})$ of a fundamental
sequence $Comp(p,z,u, v)$ for a configuration group modulo $p$.

\item [(b)] At least one of these algebraic solutions  cannot be extended to  solutions of equations corresponding to terminal groups of  resolutions for $Red_1,\ldots ,Red _k$,  $F _{R(V_{3,m_s})}$, $RQ_1,\ldots ,RQ_m$.

\item  [(c)] These solutions do not factor through 
 the terminal groups of any of the Extra PS resolutions
derived from  $Comp(p,z,u, v)$.

\end{enumerate}

In this case there is a generic certificate of level $2$ width $i$ and depth 1.
\begin{lemma} \label{main4} Conditions (a)-(c) can be described by a boolean combination of  $\exists\forall$-formulas of type (\ref{AE113}) if the theory has generators of $F_n$ as constants and 
by a boolean combination of formulas 
\begin{equation}\label{AE123}
\exists x\forall y \left( (\vee _{i=1}^m U_i(p,x)=1)\wedge (\vee _{j=1}^n V_j(p,x,y)\neq 1)\right ),
\end{equation} where $U_i(p,x)=1$ are systems of equations,  $V_j(p,x,y)$ are collections of words, 
if it does not.

\end{lemma}
\begin{proof}  If the theory has generators of $F_n$ as constants, formula (\ref{AE123}) can be rewritten in the form (\ref{AE113}) (see \cite[Section 3.1]{Imp}). Therefore it is sufficient to prove the statement for formulas (\ref{AE123}).  

Let $N$ be the  bound on the number of algebraic solutions of $U_{i,coeff}=1,$ where $Term (p, p^{(1)})=F_{R(U_{i,coeff})}$.  Conditions (a) and (b) on $p$ can be restated as negation of the following: for any   $s\leq N$ there are exactly $s$ algebraic solutions of $U_{i,coeff}=1$ and each of them can be extended to  solutions of equations corresponding to terminal groups of  resolutions for some of $Red_1,\ldots ,Red _k$,  $F _{R(V_{3,m_s})}$, $RQ_1,\ldots ,RQ_m$. 

The statement: there are at least $s$ algebraic solutions $y_1,\ldots ,y_s$ of $U_{i,coeff}=1$  can be expressed by $\exists\forall$-formula of type (\ref{AE123}) with $m=1$. Te statement that there are no $s+1$ algebraic solutions of 
$U_{i,coeff}=1$ is the negation of the statement that there are at least $s+1$ such solutions.

The statement: there are at least $s$ algebraic solutions $y_1,\ldots ,y_s$ of $U_{i,coeff}=1$ and  each $y_i$ can be extended to  solutions of equations corresponding to terminal groups of  resolutions for some of $Red_1,\ldots ,Red _k$,  $F _{R(V_{3,m_s})}$, $RQ_1,\ldots ,RQ_m$,  can be  expressed by $\exists\forall$-formula of type (\ref{AE123}).  Notice that $m$ can be greater than 1 here because each $y_i$ can be extended to  a disjunction of solutions of equations. 

Condition (c) on algebraic solutions $y_1,\ldots ,y_s$  can be also expressed by formulas of type  (\ref{AE123}).
\end{proof}
Similarly we consider resolutions of level 2 width $i$ and depth 2 and deeper resolutions of level 2 width $i$.  The sieve procedure of constructing deeper and deeper resolutions of level 2 and width $i$ stops. To prove termination of this iterative
procedure in the minimal (graded) rank case (section 1 in \cite{Sela5}), one  uses
the strict decrease in the complexity of the resolutions associated with successive
steps of the procedure.  
 In general case we will give a hint of what is involved in the proof of its termination in Section \ref{sieve}. 

With each of the  resolutions constructed along
the sieve procedure, we associate a set TSPS$(p)$, that contains the set
of specializations of the defining parameters $p$ that have a generic family
of certificates that factor through that  resolution (i.e.
a generic certificate). The sets TSPS$(p)$ associated with the
various resolutions constructed along the sieve procedure are
all in the Boolean algebra of $\forall\exists$ sets (as follows from Lemma \ref{main4}), and if a specialization $p_0$ has 
a certificate, then $p_0$ is contained in at least one of the sets TSPS$(p)$
constructed along the sieve procedures associated with the configuration groups. 

As the proof of  Lemma \ref{main4} shows,  for any specialisation $p_0$ of the defining parameters in some set TSPS$(p)$, there are finitely many combinations for the ungraded resolutions covered by a given graded resolution and not covered by another graded resolution.  These finitely many possibilities form a stratification on the set of specialisations of the defining parameters, obtained from the bases of all the graded resolutions that have been constructed in the process. A specialisation $p_0\in \Theta (p)$ if and only if it belongs to a certain strata in the combined stratification. The set of specialisations $p_0$ associated with a given combination of strictly solid families of solutions of finitely many solid limit groups can be defined using disjunctions of conjunctions of $\forall\exists$ and $\exists\forall$ predicates. 

\begin{theorem}\label{main3}
 Every definable set in $F_n$ (and its image in $\Gamma$) can be defined by the same boolean combination of
formulas (\ref {AE113}) if the language contains generators of $F$ as constants, and by a boolean combination of existential formulas and
formulas (\ref {AE123})
if the language does not have constants other than the identity element.
\end{theorem}
\begin{proof} Let us recall what we just did. Given
a fixed proof system, we collected all its certificates in a canonical collection of finitely many configuration groups.
 With each of the configuration groups we associate a sieve procedure. We started the analysis of the set of initial graded $p$-PS resolutions of level 2, by associating with
each such resolution a finite collection of reference resolutions, that collect
all the generic points (=generic families =test sequences) that factor through a given graded $p$-PS resolution
for which the corresponding pseudo-certificate (associated with the generic
point) is in fact a fake certificate. We did that by collecting all the generic families
that factor through a PS resolution for which the corresponding
pseudo-proofs are fake proofs (F1-F4), and use graded formal
limit groups, to associate a finite collection of graded resolutions
to this collection of generic families.
By construction, every value $p_0$ of the parameters $p$, for which
there exists a generic point that does not factor through one of the 
resolutions F1-F4, is necessarily in the set $\Theta (p)$ (since by Proposition \ref{genfam} it has a test
sequence (= generic family) of certificates). 

The constructed F1-F4 resolutions collect all the generic points for
which the associated pseudo-certificates are in fact fake certificates. For F1-F3, that are described by a finite disjunction of Diophantine conditions, if the pseudo proof associated with a generic point
in a fiber is fake, the pseudo proofs in the whole fiber are fake.  For others
(Extra PS), it may be that a fiber contains certificates
though the specializations associated with a generic point in the fiber are fake.
Hence,  one needs to collect all
such certificates. One does that by collecting all the pseudo certificates
in such resolutions, for which the reason that a generic pseudo certificate fails,
collapses. Such a collapse can be described by the
union of finitely many Diophantine conditions, which are therefore the same for $F_n$ and $\Gamma$, and the collection of pseudo
certificates that satisfy one of these Diophantine conditions is a Diophantine
set, with which we  associate finitely many limit groups, that we
called Collapse extra PS limit groups. One  starts the first step of the sieve
procedure with these groups.
The first step of the sieve procedure studies the structure of the obtained
Collapse extra PS limit groups (and their well separated resolutions) in comparison with
the PS resolutions (of the configuration groups) from which they were constructed.

  Since all the well separated resolutions correspond to some Diophantine conditions imposed on the previous well separated resolutions,  their construction is the same for $F_n$ and $\Gamma$. Moreover the finitely many solid limit groups that are in the bases of these resolutions are also the same for $F_n$ and $\Gamma$. A family of non-equivalent strictly solid solutions over $\Gamma$ has all the pre-images non-equivalent strictly solid.  And if two strictly solid solutions are equivalent over $\Gamma$ then w.o.p. they have some pre-images that are equivalent over $F_n.$

  Similarly we consider well separated resolutions of level 2 width i and depth 2 and deeper
resolutions of level 2 width i. The construction is very technical and we will hint in the next section why  the sieve procedure of constructing deeper and deeper
resolutions of level 2 and width $i$ stops. Theorems \ref{prelim} and \ref {5.10} ensure that the procedure for $F_n$ and $\Gamma$ goes the same way. This proves Theorem \ref{main11}.

Dealing similarly with resolutions of a level higher than 2, one proves Theorem \ref{main3}.
\end{proof} 

Now Theorem \ref{main3} implies Theorem \ref{main}.
\section{Termination of the sieve procedure in general case} \label{sieve} 

We recall the notion of complexity of a resolution  ($Cmplx (Res)$) at follows:
$$Cmplx(Res) =(dim (Res) + factors (Res), (size (Q_1),\ldots ,size (Q_m), ab(Res)),$$
where $factors (Res)$ is the number of freely indecomposable, non-cyclic terminal factors (this component only appears in graded resolutions),  and   $$(size (Q_1),\ldots, size (Q_m))$$ is the regular size of the resolution \cite{KMbook}. The complexity is a tuple of numbers which we compare in the left lexicographic order. The number
 $dim (Res) + factors (Res)$ is called the Kurosh rank of the resolution.

For convenience, we adopt the terminology of \cite{Sela5.2}, which the third author is more accustomed to and which provides terminology for structures some of which were  used but left unnamed in \cite{KMel}.

Suppose a configuration group is fixed, and we are considering the set $True (\Theta)_2$ ($True (\Theta)_{2\Gamma}$. On the first step of the sieve procedure one  starts by looking at the Zariski closure of all the pseudo certificates
that can be extended to specializations of a given Collapse extra PS limit
group. This Zariski closure is  contained in the variety associated
with the configuration group we started with. If the limit
group that is associated with the Zariski closure is a proper quotient of the configuration group, we replace the configuration group with the limit
group that is associated with the Zariski closure.  The descending
chain condition  for limit groups guarantees that such a
replacement can be done only finitely many times along the entire procedure.

For the rest of the first step we  assume that there is no decrease
in the Zariski closure of the Diophantine set of pseudo certificates.
By construction, there is a natural map from the completion of the PS
resolution we started with to a given Collapse extra PS limit group. One 
then looks at the multi-graded resolutions of the given Collapse extra
PS limit group with respect to the (images of the) non-QH, non-abelian
vertex groups in the top level of the completion of the PS resolution we
started with (more presicely, one uses auxiliary resolutions to do it with respect
to the images of the non-QH, non-abelian vertex groups and edge groups
in the graded abelian JSJ decomposition of the original configuration group).
 We use  tight enveloping NTQ groups \cite{KMel} or \cite[Section 3.5.3]{KMbook} and their resolutions (similar to core resolutions that were presented in \cite[Section 4]{Sela5}.  The complexity of each such multi-graded core resolution can be bounded by the complexity of the top level of the completion of the initial PS resolution, which, in turn, essentially corresponds to the complexity of the graded abelian JSJ decomposition of the configuration group.
 The image of the original configuration group in the terminal limit group of an obtained multi-graded resolution
is a proper quotient of the initial configuration group we started with. If there is a decrease in the complexity of an obtained multi-graded resolution
(in comparison with the complexity of the top level of the completion of
the corresponding PS resolution), then we consider a so-called developing resolution, which
is a completed  resolution
induced by the original configuration group from the given multi-graded
resolution, followed by a $p$-graded resolution from the taut M-R diagram of the image of the configuration group in
the terminal level of the given multi-graded resolution. (In \cite{KMel} there is a similar construction that is called a {\em projective image}).
 The so-called anvils  are maximal limit quotients of the amalgamated product of the completion
of the given multi-graded resolution with the completion of the developing resolution 
amalgamated along the common part of the developing resolution.

If there is no decrease in the complexity of an obtained multi-graded resolution,
it is guaranteed that the structure of the tight enveloping (core) resolution associated
with such a multi-graded resolution is identical to the structure of the top
level of the completion of the initial PS resolution. In this case we continue to multi-graded resolutions associated with lower
levels in the completion of the initial PS resolution. If there is
a decrease in complexity, one associates anvils and developing resolutions with
the obtained multi-graded resolutions, precisely as we constructed them in
case of a complexity decrease in the top level.

Suppose there exists a multi-graded resolution in which the complexity of the tight enveloping (core) resolutions remains unchanged across all levels. In this case, we enlarge the collection of resolutions under consideration and examine all graded resolutions (with respect to the defining parameters 
$p$
) whose formed part of the tight enveloping (core) resolutions at each level shares the same structure as the formed part of the abelian decompositions associated with the corresponding levels in the completion of the initial PS resolution. Here, the formed part of an abelian decomposition refers to the portion consisting of its QH and abelian vertex groups. (see \cite[Definition 4.17]
{Sela5}).
One sets the ambient resolution to be the anvil. For each such graded
resolution we look at the core resolutions associated with its various parts.
If for some part the complexity of a core resolution associated with it is
strictly smaller than the complexity of the corresponding level of the completion
of the PS resolution, we set the ambient resolution to be a carrier,
and the resolution induced by the image of the original PS limit group to
be the developing resolution. If no reduction of complexity occurs, we set
the resolution associated with the image of the original PS limit group to
be a sculpted resolution, and set the developing resolution to be the ambient
resolution.
To conclude the first step, we associate with each developing resolution
(that is naturally mapped into the anvil), a finite collection of reference
resolutions F1-F4, in the same way we associated these with the original PS resolutions.
With each anvil we further associate a Collapse extra PS limit
groups, that collects those specializations of the bad reference fibers, for
which the reason for the failure of a generic pseudo-proof collapses, in
the same way we associated these with the original PS resolutions.

One continues with the sieve procedure iteratively, taking into account all the constructions that were made along the different branches of the
procedure in all the previous steps. We suppose that with each anvil and
its associated developing resolution that were constructed in step $n-1$ of
the sieve procedure, there is an associated positive integer called width that
denotes the number of  (nested) sequences of core resolutions associated with the anvil.
Width$(n-1)=1$ if and only if no sculpted resolution is associated with the
anvil. In case width $(n-1)>1$, with the (nested) collection of sequences of core resolutions associated with the anvil, there are
associated sculpted resolutions and possibly carriers. We start the general
step of the sieve procedure with the (finite) collection of anvils constructed
at the previous step of the procedure, and their associated Collapse extra
PS limit groups.

The general step of the procedure is essentially similar to the first step,
though it is technically more complicated since it has to take into account all the
previous steps of the procedure. Like in the first step, we start by looking at the
Zariski closure of all the virtual proofs that can be extended to specializations
of a given Collapse extra PS limit group. If this Zariski closure is
strictly contained in the variety we started the first step with, we continue
by replacing the configuration group with the limit group that is associated with
the Zariski closure (which is in this case a proper quotient of the configuration 
group we started with), and start again the first step of the sieve procedure
with this Zariski closure. As we have already noted, by the descending chain condition for limit
groups this can happen at most finitely many times along a branch of the
sieve procedure.
By construction, there are  maps from the completions of the
various resolutions constructed in previous steps of the procedure into the
Collapse extra PS resolution. As in the first step, we proceed by analysing
the multi-graded resolutions of the given Collapse extra PS limit group
with respect to the images of these resolutions. If there is a decrease in the
complexity of an obtained multi-graded resolution (in comparison with the
complexity of the top level of the completion of the corresponding resolution),
we continue by associating with the multi-graded resolution a finite
collection of anvils and developing resolutions in a similar way to what we
did in the first step.
If there is no decrease in the complexity of an obtained multi-graded resolution,
then the structure of the tight enveloping (core) resolution linked to
such a multi-graded resolution is the same as  the structure of the tight enveloping (core)
resolution associated with the corresponding multi-graded resolution that
was constructed in the previous step of the procedure. Then we continue
to multi-graded resolutions associated with lower levels in the anvil
we started the general step with. If there is a decrease in complexity we link
anvils and developing resolutions with the obtained multi-graded
resolutions, the same way as we constructed them in case of a complexity decrease
in the top level.

Assume there exists a multi-graded resolution in which the complexity of the core resolutions remains unchanged at all levels. In that case, we proceed to analyse the resolutions associated with the next sculpted resolution - if such resolutions are linked to the original anvil.

If, at any point, we observe a drop in the complexity of a resolution tied to one of the algebraic envelopes (i.e., nested sequences of core resolutions) associated with the original anvil, we construct a finite collection consisting of anvils, developing resolutions, sculpted resolutions, and carriers - mirroring the construction carried out in the initial step.

On the other hand, if no such decrease in complexity occurs, we augment the collection of algebraic envelopes related to the starting anvil. We then either:
\begin{itemize}
\item declare the new algebraic envelope to be a carrier for a sculpted resolution - provided that the resolution linked to it leads to a drop in complexity, or

\item define a new, additional sculpted resolution.
\end{itemize}
In both scenarios, we construct a finite system of anvils and developing resolutions, associating them with the newly introduced algebraic envelopes, carriers, and sculpted resolutions.

As in the first step, each developing resolution (naturally mapping into the anvil) is assigned a finite collection of reference resolutions, constructed in the same manner as those linked to the original PS resolution. Additionally, with each anvil, we associate a collection of Collapse Extra PS Limit Groups, which capture the specializations of the problematic reference fibers - those for which the generic pseudo-proof collapses - again following the same procedure used in the initial stage.

To prove termination of the iterative
procedure in the minimal graded rank case (section 1 in \cite{Sela5}), one  uses
the strict decrease in the complexity of the resolutions associated with successive
steps of the procedure.  
In
the procedure used to construct the parametric $\forall\exists$ tree  one  proves termination by combining the decrease
in either the Zariski closures or the complexities of the resolutions and
decompositions associated with successive steps of the procedure.

In the general step of the sieve procedure, it may occur that neither the Zariski closures nor the complexities of the various core and developing resolutions linked to a given anvil show any reduction. To handle this situation, the procedure introduces supplementary algebraic envelopes, sculpted resolutions, and carriers, all of which are assigned to the anvil in question. Consequently, beyond the reasoning used to demonstrate the termination of the parametric $\forall\exists$-tree construction, an additional step is required to establish the termination of the sieve procedure: namely, one must show that there is a uniform bound on the number of sculpted resolutions that can be associated with an anvil. The method for proving this bound parallels the approach used to establish a global bound on the number of rigid and strictly solid families of specializations of rigid and solid limit groups, as developed in \cite{KMel} and the first two sections of \cite{Sela}.

\section{Appendix: On the density parameter.}

In \cite{Mass1} Massalha extended the main result of \cite{KhS} to the case $d<1/2$. That is, he proved.

\begin{theorem}(\cite[Theorem 10.1]{Mass1})
Let $d<1/2$. Let $V(\bar{x}, e_1, \ldots, e_n) = 1$ be a system of equations over $\mathbb{F}_n=\langle e_1, e_2, \ldots, e_n\rangle$. Suppose $\Gamma_\ell$ is a random group of density $d$ at length $\ell$ and $\pi_\ell: \mathbb{F}_n \rightarrow \Gamma_\ell$ is the canonical quotient map.

Then, every solution $\bar b_\ell$ of $V(x,e_1, \ldots, e_n) = 1$ in $\Gamma_\ell$ is the image of a solution $\bar c_\ell$ of $V(x,e_1, \ldots, e_n) = 1$ in
$\mathbb{F}_n$, under the canonical quotient maps, i.e. $\pi_\ell(\bar c_\ell)=\bar b_\ell$, with probability approaching 1 as $\ell$ goes to infinity.    
\end{theorem}

Massalha uses the general framework of \cite{KhS}, but in addition he cleverly exploits some results in hyperbolic geometry. We explain the common strategy and the ideas of the above result, adapting it to our language. 

The proof splits in two main parts. Note that this is not very different from Gromov's original approach (see \cite{Olli} Ollivier and Gromov \cite{Gromov}) for proving that a random group is hyperbolic. The first part is to prove a local-to-global result (in order to bound the number of van Kampen diagrams), and the second part is to calculate the probability that some (of boundedly many) van Kampen diagram is fulfillable. 

We recall Gromov's local-to-global principle (suitably adjusted). Intuitively, it says that in order to check the hyperbolicity of a (length) metric space, one needs to check the hyperbolicity locally, i.e. in bounded balls of fixed diameter. Ollivier calls the following theorem Cartan-Hadamard-Gromov-Papasoglu theorem. For the notation to make sense we note in passing that $|D|$ denotes the number of faces and $|\vartheta D|$ the length of the exterior boundary of the van Kampen diagram $D$.

\begin{theorem}
For each $\alpha>0$, there exist an integer $K:=K(\alpha)\geq 1$ and an $\alpha'>0$ such that, if a group is given by relations of length $\ell$, for some $\ell$, and if any reduced van Kampen diagram with at most $K$ faces satisfies $|\vartheta(D)|\geq \alpha \ell |D|$, then any reduced van Kampen diagram satisfies $|\vartheta(D)|\geq \alpha' \ell |D|$
\end{theorem}

In our case the local-to-global result translates to showing that a ``short non-free" solution of a system of equations in a random group has bounded length. Our proof \cite{KhS} bounds the length of such solution linearly with respect to the length of the relators. Assuming that $d<1/16$ and consequently that a random group at this density is alsmost surely an eighth-group the linear bound works well. On the other hand, Massalha \cite{Mass1} observed that allowing the bound to grow double logarithmically does not hurt the second part of the proof (i.e. calculating the probabilities) and in addition one can work out the proof without assuming the special properties of eighth-groups. In particular, one can increase the density parameter to $d<1/2$. 

The main technical tool for proving the above results in both cases is the shortening argument. The shortening argument has been used in many different ways and various assumptions, but until \cite{KhS} it has never been used to a sequence of hyperbolic groups with unbounded hyperbolicity constant. Recall that a random group of length $\ell$, $\Gamma_\ell$, is, with overwhelming probability, $\delta_\ell$-hyperbolic, where $\delta_\ell$ grows linearly with $\ell$. This predicament did not allow us to use the shortening argument as a black box, but we had to tweak its details. The idea that the shortening argument can still be used and the technical details that allowed us to use it was the part, in our proof, that required most of the genuine ideas. More formally we have:

\begin{theorem}[Kharlampovich-Sklinos \cite{KhS}]
Let $d<1/16$. Let $V(x)=1$ be a system of equations and $G_V=\langle x \ | \ V(x) \rangle$ the corresponding group.  Suppose $h_\ell:G_V\rightarrow \Gamma_\ell$ be a short non-free morphism, then $|h_\ell|\leq K\ell$ with overwhelming probability. 
\end{theorem}

\begin{theorem}[Massalha \cite{Mass1}]
Let $d<1/2$. Let $V(x)=1$ be a system of equations and $G_V=\langle x \ | \ V(x) \rangle$ the corresponding group.  Suppose $h_\ell:G_V\rightarrow \Gamma_\ell$ be a short non-free morphism, then $|h_\ell|\leq (Kln^2\ell)\ell$ with overwhelming probability. 
\end{theorem}

We repeat that in order to get the linear growth in the work of Kharlampovich-Sklinos an extensive use of the properties of eighth groups was made, thus the stricter density parameter. In the work of Massalha the growth is relaxed to double logarithmic and the gain is the larger density parameter. 

Finally, the second part of the proof, i.e. calculating probabilities, is not affected much. In either case, linear or double logarithmic, the probability that some diagram (from boundedly many) is fulfillable goes to $0$ and this completes the proof. We explain how our lemmas can be adapted in the case the growth is double logarithmic.

\begin{lemma} (see \cite[Lemma 7.4]{KhS})\label{NumberofDiagrams}
The total number of families of triangular abstract van Kampen diagrams with $f$ faces labeled by $n\leq f$ different realators of length $\ell$, and whose sides have length less than $(Kln^2\ell)\ell$, for some constant $K$, is bounded by $$\bigl(K(ln^2l)2^{13}\ell^7\bigr)^{f+2q}S(f,n)$$ where $S(f,n)$ is the Stirling number. 
\end{lemma}

The proof of the above lemma is identical to the proof of \cite[Lemma 7.4]{KhS} except that we have to remember that the length of each side of a triangular abstract van Kampen diagram is less than $K(ln^2\ell)\ell$ instead of less than $K\ell$.

\begin{prop}(see \cite[Proposition 7.2]{KhS})\label{NumberofFaces}
Let $\Sigma(x)=1$, be a triangular system of equations. Suppose that the length of a short non-free solution of $\Sigma(x)=1$ (in a random group of length $\ell$ and density $d<1/2$) is bounded by $K(ln^2\ell)\ell$, for some constant $K$, with overwhelming probability. 

Then the number of faces of the family of abstract van Kampen diagrams that corresponds to the solution is, with overwhelming probability, bounded by $R(ln^2\ell)$, for some constant $R$ that depends only on the number of variables, the number of equations, and $d$.     
\end{prop}

The last piece needed to complete the puzzle comes from calculating the probabilities of fulfilling a family of non-filamentous abstract van Kampen diagrams. This result is independent of how fast a short non-free solution grows with $\ell$.

\begin{lemma}(\cite[Corollary 7.14]{KhS})\label{ProbSingleDiagram}
Let $d<1/2$. Let $\mathbb{D}$ be a family of non-filamentous abstract van Kampen diagrams. Suppose $\mathbb{D}$ has $f$ faces that should be filled in by $n\leq f$ relations. In addition, suppose that $\mathbb{D}$ is decorated by a non-free solution of a system of equations $\Sigma(x)=1$. 

Then the probability that $\mathbb{D}$ can be fulfilled by a random group of level $\ell$ and density $d$ in a way that the decoration is satisfied is less than: 
$$(2m-1)^{-\ell(1/2-d)}$$
\end{lemma}

Finally, we can bring these results together to get:

\begin{theorem}(\cite[Proposition 7.15]{KhS})
Let $d<1/2$. Let $\Sigma(x)=1$ be a system of equations. The probability that a random group of level $\ell$ and density $d$ satisfies some family of decorated abstract van Kampen diagrams whose decoration is induced by a short non-free solution of $\Sigma(\bar{x})=1$ goes to $0$ as $\ell$ goes to $\infty$.    
\end{theorem}
\begin{proof}
The probability of the statement is bounded by the probability of Lemma \ref{ProbSingleDiagram} multiplied by the appropriate number of possible abstract van Kampen diagrams for which a bound is given by the combination of Proposition \ref{NumberofDiagrams} and Lemma \ref{NumberofFaces}. Hence, we need to bound
the function $$F(\ell)=\frac{\sum_{f=1}^{Rln^2\ell}\sum_{n=1}^{f}\bigl(K(ln^2l)2^{13}\ell^7\bigr)^{f+2q}S(f,n) }{(2m-1)^{\ell(1/2-d)}}\leq $$
$$ \frac{\ln ^4 \ell\sum_{n=1}^{\ln ^3\ell}(K(\ln ^2\ell)2^{13}\ell ^7)^{\ln ^3\ell} S(\ln ^3\ell ,n)}{(2m-1)^{\ell(1/2-d)}}\leq  \frac{\ln ^7 \ell(K(\ln ^2\ell)2^{13}\ell ^7)^{\ln ^3\ell} B_{\ln ^3\ell}}{(2m-1)^{\ell(1/2-d)}},$$ where the Bell number $B_{\ln ^3\ell}$ is bounded by $(\frac{0.792 \ln ^3\ell}{\ln (\ln ^3\ell )})^{\ln ^3\ell}.$
Therefore for sufficiently large $\ell$, $F(\ell )$ is bounded from above by
the function $$\frac{{\ell}^{\ln ^4\ell}}{R^{\ell}},$$ where  $R=(2m-1)^{(1/2-d)}.$ We will now show that this function approaches zero as $\ell$ approaches infinity.
$$\frac{{\ell}^{\ln ^4\ell}}{R^{\ell}}=\frac{{e}^{\ln ^5\ell}}{R^{\ell}}=e^{\ln ^5\ell -(\ln R)\ell}.$$
This approaches zero because $\ln ^5\ell -(\ln R)\ell$ approaches $-\infty .$
\end{proof}

\end{document}